\documentclass[12pt, reqno]{amsart}
\usepackage{amsfonts,amsmath,amsthm,oldgerm,amssymb,amscd}

\usepackage[all,cmtip]{xy}
\usepackage{fancyhdr}
\usepackage{psfrag}
\usepackage{graphicx}
\usepackage{amsbsy}
\usepackage{indentfirst}
\usepackage{float}
\usepackage{latexsym}
\usepackage{graphics}
\usepackage{verbatim}
\usepackage{fixmath}
\newcommand{\R}{{\mathbb R}}
\newcommand{\Z}{{\mathbb Z}}
\newcommand{\C}{{\mathbb C}}

\newcommand{\A}{{\mathbf A}}
\newcommand{\G}{{\mathbf G}}
\newcommand{\bs}{{\mathbf s}}

\newcommand{\fS}{{\mathfrak S}}

\newcommand{\cE}{{\mathcal E}}

\newcommand{\cG}{{\mathcal G}}

\newcommand{\cH}{{\mathcal H}}

\newtheorem{theorem}{Theorem}
\newtheorem{lemma}{Lemma}
\newtheorem{proposition}[theorem]{Proposition}

\theoremstyle{definition}

\numberwithin{definition}{section}

\numberwithin{equation}{section}

\theoremstyle{remark}
\newtheorem*{remark}{Remark}
\newtheorem*{acknowledgement}{Acknowledgement}

\begin{document}

\title[Solvable base change]{Solvable base change}

\author{L. Clozel}

\address{Universit\'e de Paris Sud, Math\'{e}matiques, B\^{a}t. 307, F-91405 Orsay Cedex.}  \email{laurent.clozel@math.u-psud.fr}

\author{C.~S.~Rajan}

\address{Tata Institute of Fundamental  Research, Homi Bhabha Road,
Bombay - 400 005, INDIA.}  \email{rajan@math.tifr.res.in}


\begin{abstract} We determine the image and the fibres for solvable base change.

\end{abstract}

\maketitle

\section{Introduction}
The reciprocity conjectures formulated by Langlands give a
parametrization of cusp forms associated to $GL_n$ over a global field
$K$ by $n$-dimensional complex representations of the Langlands group attached
to $K$. The Langlands group, whose existence is yet to be shown, is a
vast generalization of the absolute Galois group or the Weil group of
$K$, and can be considered in analogy with these latter groups.   In
this analogy, the theory of
base change amounts to restriction of parameters on the Galois
theoretic side. 

For cyclic extensions of number fields of prime degree, 
the existence 
and the characterization of the image and fibres of base change for
$GL(2)$ was
done by Langlands (\cite{L2}) following earlier work of Saito and
Shintani. This was used by Langlands to establish Artin's conjecture
for a class of octahedral two dimensional representations of the
absolute Galois group of a number field. The work of Saito, Shintani
and Langlands was generalized by Arthur and Clozel to $GL(n)$, for
all $n$.  In \cite{AC}, they proved the existence and characterized
the image of the base change transfer for cyclic extensions of number
fields of prime degree. However, the proof in the general, cyclic case contained a mistake. 

The theorem of Langlands, Arthur and Clozel, 
gives inductively the existence of the base
change transfer corresponding to a solvable extension of number fields
for $GL_n$. The problem of characterizing the image and fibres of base
change for cyclic extensions of non-prime
degree was considered by Lapid and Rogawski in \cite{LR}. This led
them to conjecture the non-existence of certain types of cusp forms on
$GL(n)$, and they proved this conjecture when $n=2$. (The general cyclic case has since been settled by other means, see Labesse \cite{Lab}.)     It was shown in
\cite{R} that the conjecture of Lapid and Rogawski allows a
characterization of the image and fibres of the base change map for
solvable extensions of number fields. In this article, our main aim is
to prove the conjecture of Lapid and Rogawski for all $n$. 

In order to make this paper more self- contained, we have included here complete  proofs of the theorem of Lapid and Rogawski characterising  cusp forms on GL(n) whose Galois conjugate by a generator of a cyclic Galois group differs from the original form by an Abelian twist (Theorem 2), which was in their paper conditional on Theorem 1; and of the theorem of one of us (Rajan) characterising the image and fiber of base  change in a solvable extension (Theorem 3.)

\subsection{Main theorem}
For a number field $F$, let $\A_F$ denote the adele ring 
of $F$ and $C_F$ the group
of idele classes of $F$. Given a representation $\pi$ of $GL(n, \A_F)$
and $\sigma$ an automorphism of $F$, define $^{\sigma}\pi$ to be the
representation $g\mapsto \pi(\sigma^{-1}(g))$ for $g\in GL(n, \A_F)$.  
Given an extension $E/F$ of number fields, if $\omega$ is an idele
class character of $E$, denote by $\omega_F$ its restriction to the
idele class group $C_F$ of $F$. If $\omega$ is a character of $C_F$,
define $\omega^E=\omega\circ N_{E/F}$, where $N_{E/F}:C_E\to C_F$ is
the norm map on the idele classes of $E$. 

Let $E/F$ be a cyclic extension of number fields of degree $d$
and $\sigma$ denote a generator for ${\rm Gal}(E/F)$.  By 
$\varepsilon_{E/F}$, we mean 
an idele class character of $C_F$ corresponding to the
extension $E/F$, i.e., a character of $C_F$ of order $d$, 
vanishing on the subgroup of norms $N_{E/F}(C_E)$ coming from $E$. 

The primary aim of this paper is to establish the following conjecture
of Lapid and Rogawski (\cite{LR}) for $GL(n)$ for all $n$, proved by
them for $GL(2)$: 
\begin{theorem}[Statement A, p.178, \cite{LR}]\label{nonexistence}
Let $E/F$ be a cyclic extension of number fields of degree $d$
and $\sigma$ denote a generator for ${\rm Gal}(E/F)$.  Let $\omega$ be an
idele class character of $E$, such that its restriction to $C_F\subset
C_E$ is  $\varepsilon_{E/F}$.  Then there does not exist any
 cuspidal automorphic representation $\pi$ of
  $GL(n, \A_E)$ such that
\begin{equation}\label{lrconj}
^{\sigma}\Pi\simeq \Pi\otimes \omega.
\end{equation}
  \end{theorem}
  
  We note that the theorem is obvious if $d$ does not divide $n$, as one sees by considering the restriction to $C_F$ of the central character of $\Pi$.

\subsection{Galois conjugate cusp forms up to twisting by a character}
From Theorem \ref{nonexistence}, Lapid and Rogawski derive a structure
theorem for cusp forms $\pi$  on $GL(n, \A_E)$, whose Galois conjugate
$\sigma(\pi)$ differs from $\pi$ up to twisting by a character
$\omega$.

Given a cyclic extension
$E/F$ of number fields and an automorphic representation $\pi$ of
$GL(n, \A_F)$, denote by $BC_F^E(\pi)$ the base change lift of $\pi$
to an automorphic representation of   $GL(n, \A_E)$. For a cusp form
$\eta$ on  $GL(n, \A_E)$, let $AI_E^F(\eta)$ denote the automorphic
representation of  $GL(nd, \A_F)$, where $d={\rm deg}(E/F)$, the
existence of which was proved for $nd=2$ in \cite{LL}, and for general
$n,d$ in \cite{AC}. For a field $F$ let $G_F$
denote the absolute Galois group ${\rm Gal}(\bar{F}/F)$ where
$\bar{F}$ denotes an algebraic closure of $F$. At the level of Galois
representations, base change corresponds to the restriction of
representations from $G_F$ to $G_E$, and 
automorphic induction corresponds to the induction of
representations of $G_E$ to $G_F$. 
\begin{theorem}[Statement B, p. 179, \cite{LR}]\label{galoisconjugate}
Let $E$ be a number field and $\sigma$ an automorphism of $E$ of order
$d$. Let $F$ be the field left fixed by the subgroup of automorphisms
of $E$ generated by $\sigma$. Let $\omega$ be an idele class character
of $E$ and $\pi$ be a cuspidal representation of $GL(n, \A_E)$ such
that $\sigma(\pi)\simeq \pi\otimes \omega$. Let $K/F$ be the extension
corresponding to the character $\omega_F$ of $C_F$, and let
$L=KE$. Then
\begin{enumerate}
\item $K\cap E=F$ and $[K:F]$ divides $n$. Let $r=n/[K:F]$ and let
  $\tau$ be the unique extension of $\sigma$ to $L$ trivial on
  $K$. 

\item There exists a cuspidal representation $\pi_0$ of $GL(r, \A_K)$
  such that 
\[\pi=AI_L^E(BC_{L/K}(\pi_0)\otimes \psi),\]
where $\psi$ is a Hecke character of $L$ such that
$\tau(\psi)\psi^{-1}=\omega^L$. 

\item Conversely, given $\pi_0$ and $\psi$ as in (ii), the
  representation 
\[\pi= AI_L^E(BC_{L/K}(\pi_0)\otimes \psi)\]
satisfies $\sigma(\pi)\simeq \pi\otimes \omega$. However $\pi$ need
not be cuspidal. 
\end{enumerate}
\end{theorem}

\subsection{Solvable base change} The following theorem characterizing
the image and fibres of the base change transfer for solvable
extensions of number fields was established in \cite{R}: 
\begin{theorem}\label{sbc}
Let $E/F$ be a  solvable extension of number fields, and let $\Pi$ be
a unitary, cuspidal automorphic representation of  
$GL_n(\mathbf{A}_E)$. 
\begin{enumerate}
\item Suppose $\Pi$  is ${\rm
Gal}(E/F)$-invariant.  
Then there exists a ${\rm Gal}(E/F)$-invariant Hecke character
$\psi$ of $E$, 
 and a cuspidal automorphic representation $\pi$ of  
$GL_n(\mathbf{A}_E)$ such that 
\[ BC_{E/F}(\pi) \simeq \Pi\otimes \psi.\]
Further $\psi$ is unique up to base change to $E$ of a Hecke
character of $F$. 

\item Suppose there exist cuspidal automorphic
representations $\pi,~\pi'$ of  $GL_n(\mathbf{A}_F)$
such that,
\[ BC_{E/F}(\pi)=BC_{E/F}(\pi')=\Pi.\]
Then there exists a character $\chi$ of $C_F$ corresponding via class
field theory to a character of ${\rm Gal}(E/F)$, such that 
\[\pi'\simeq \pi\otimes \chi.\]
Morever if $\chi$ is non-trivial,  the representations $\pi$ and
$\pi\otimes \chi$ are distinct. 

\end{enumerate}
\end{theorem}
Theorem \ref{sbc} follows by an inductive argument from Theorem
\ref{galoisconjugate}. 

\begin{remark}
Suppose $E/F$ is a solvable extension with the
property that invariant idele class characters of $E$ descend to
$F$. Then any invariant, unitary, cuspidal automorphic representation
of  $GL_n(\mathbf{A}_E)$ lies in the image of the base change map
$BC_{E/F}$.
In particular, we recover the classical formulation that invariant,
unitary, cuspidal  
automorphic representations descend if $E/F$ is cyclic. 
\end{remark}

\begin{remark} The motivation for this theorem stems from the following analogous Galois theoretic situation: let $E/F$ be a Galois extension of number fields, and $\rho: G_E\to GL_n(\C)$ an irreducible representation of $G_E$. Suppose that $\rho$ is invariant under the action of $G_F$ on the collection of representations of $
G_E$. By an application of Schur's lemma, it can be seen that $\rho$ extends as a projective representation, say $\tilde{\eta}$ to $G_F$. By a theorem of Tate on the vanishing of $H^2(G_F, \C^\times)$ (\cite{S1}), this representation can be lifted to a linear representation $\eta$ of $G_F$. This implies that  $\rho\otimes \chi$ descends to a representation of $G_F$ for some character $\chi$ of $G_E$.  
\end{remark}

\begin{acknowledgement}

Clozel's work was partially supported by the National Science foundation under Grant No. DMS-1638352.

We thank J.-L. Waldspurger for useful correspondence.

\end{acknowledgement}

\section{Proof of Theorem \ref{galoisconjugate}}
In this section, assuming the validity of Theorem \ref{nonexistence},
we give a proof of Theorem \ref{galoisconjugate}
 modifying
the arguments given in \cite{LR}. 

Let $E/F$ be a cyclic extension of degree $d$, and let $\sigma$ denote
a generator for ${\rm Gal}(E/F)$. We want to classify the idele class
characters $\omega$ of $E$ and cusp forms $\pi$ on $GL(n,\A_E)$
satisfying the condition $\sigma(\pi)\simeq \pi\otimes \omega$. 
At the level of central characters, this implies
$\sigma(\chi_{\pi})=\chi_{\pi}\omega^n$, where $\chi_{\pi}$ denotes
the central character of $\pi$. In particular this implies that the
restriction $\omega_F$, of $\omega$ to $C_F$, has finite order
dividing $n$. Let $K/F$ be the cyclic extension of $F$ corresponding
to $\omega_F$. 

We first show that $K\cap E=F$.  We prove this by induction on the
pair $(d,n)$,  assuming the validity of the claim 
for all extensions $E/F$ of degree less than $d$,  or  for cusp
forms on $GL_m$ for $m<n$. When $n=1$ or $d=1$, the assertion is
clearly true. So, we now assume that $d>1$ and $n>1$. 

We initially rule out the following case: $K\subset E$ and $K$ is not
equal to $F$. The case $E=K$ is ruled out by Theorem
\ref{nonexistence}. Let 
$F\subset K\subset E'\subset E$ be an extension of fields, such that 
the degree of $E/E'$ is $p$, for some rational prime $p$. 
Let $m=[E':F]$. The group ${\rm Gal}(E/E')$ is generated by
$\sigma^m$. We observe, 
\[\sigma^m(\pi)\simeq \sigma^{m-1}(\pi\otimes \omega)\simeq
\cdots\simeq \pi\otimes
\psi,\]
where $\psi=\omega\sigma(\omega)\cdots\sigma^{m-1}(\omega)$.   The
restriction $\psi_{E'}$ of $\psi$ to $E'$ satisfies,  
\[ \psi_{E'}=\psi_F\circ N_{E'/F}. \]
Since $K\subset
E'$, $\psi_{E'}$ is trivial. Consequently, there exists an idele class
character $\eta$ of $C_E$ such that $ \eta/\sigma^m(\eta)=\psi.$
Hence, 
\[\sigma^m(\pi\otimes \eta)\simeq \sigma^m(\pi)\otimes
\sigma^m(\eta)\simeq \pi\otimes \psi \otimes \sigma^m(\eta)\simeq \pi\otimes
\eta.\] 
By the descent theorem for cusp forms invariant with respect to a
cyclic extension of prime degree (\cite{AC}), there exists a cuspidal
automorphic representation $\rho'$ of $GL_n(\A_{E'})$ such that 
\[ BC_{E'}^E(\rho')=\pi\otimes \eta.\]
Let $\sigma'$ denote the restriction of $\sigma$ to $E'$. Then, 
\[  BC_{E'}^E(\sigma'(\rho'))\simeq \sigma(\pi\otimes \eta)\simeq
\pi\otimes \omega \otimes \sigma(\eta)\simeq  BC_{E'}^E(\rho')\otimes 
\sigma(\eta)\eta^{-1}\omega.\]
We observe now that $\sigma(\eta)\eta^{-1}\omega$ is invariant under ${\rm
  Gal}(E/E')=<\sigma^m>$. From the definition of $\psi$, it follows that 
$\sigma(\psi)\psi^{-1}=\sigma^m(\omega)\omega^{-1}$. Hence, 
\[\sigma^m(\sigma(\eta)\eta^{-1}\omega)=\sigma(\sigma^m\eta)\sigma^m(\eta)^{-1}
\sigma^m(\omega)=
\sigma(\eta)\sigma(\psi)^{-1}\psi\eta^{-1}\sigma(\psi)\psi^{-1}\omega
=\sigma(\eta)\eta^{-1}\omega.\]
Hence there exists an idele class character $\theta$ of $E'$ such that
$\theta\circ N_{E/E'}= \sigma(\eta)\eta^{-1}\omega$. We have, 
\[  BC_{E'}^E(\sigma'(\rho'))\simeq BC_{E'}^E(\rho'\otimes \theta).\]
From the characterization of the fibres of base change for cyclic
extensions of prime degree, there exists  an idele class character
$\chi$ of $E'$ such that 
\[ \sigma'(\rho')\simeq \rho'\otimes \chi.\]
The base change of the character $\chi$ to $E$ is
$\sigma(\eta)\eta^{-1}\omega$. Hence $\chi_F^p=\omega_F$, and $\chi_F$
defines an extension $K'$ of $F$ containing $K$. Since the degree
$[E':F]$ is less than that of $[E:F]$, by the inductive hypothesis
$K'\cap E'=F$. This implies that $K=K\cap E=F$. Hence we have ruled
out the case that $K\subset E$ and $K\neq F$. 

Let $N=K\cap E$. We want to show $N=F$. Assume now that $\omega_F$ is
not trivial. The field $K$ associated to $\omega_F$ is a non-trivial
cyclic extension of $F$. 
By what we have shown above,
$K\cap E$ is a proper subfield of $K$. Let $K'\subset K$ be a subfield
of $K$ containing $N$, such that its degree over $N$ is a rational
prime $p$.   We have, 
\[\pi\simeq \sigma^d(\pi)\simeq
 \pi\otimes\omega\sigma(\omega)\cdots\sigma^{d-1}(\omega)\simeq
 \pi\otimes \omega_F\circ N_{E/F}.\]
By class field theory, the character $\omega_F\circ N_{E/F}$
corresponds to the cyclic extension $L=KE$ over $E$. The field
$E'=EK'$ is an extension of $E$ of degree $p$ contained inside
$L$. The isomorphism $\pi\simeq \pi\otimes\omega_F\circ N_{E/F}$
implies an isomorphism  $\pi\simeq \pi\otimes(\omega_F\circ
N_{E/F})^k$ for any natural number $k$. Taking $k=[L:E']$, we have 
$\pi\simeq \pi\otimes \varepsilon_{E'/E}$, where $\varepsilon_{E'/E}$ is an
idele class character of $E$ corresponding via class field theory to
the extension $E'/E$. By the characterization of automorphic induction
(\cite{AC}), there exists a cusp form $\pi'$ on $GL(n/p, \A_{E'})$
such that 
\[\pi\simeq AI_{E'}^E(\pi').\]
Let $\sigma' \in Gal(E'/F)$ be an extension of $\sigma$ to $E'$. Then
\[AI_{E'}^E(\sigma'(\pi'))\simeq \sigma(\pi)\simeq \pi\otimes
\omega\simeq 
AI_{E'}^E(\pi'\otimes \omega'),\]
where $\omega'=\omega\circ N_{E'/E}$ is the base change of $\omega$
to $E'$. From the characterization of the fibres of automorphic
induction with respect to a cyclic extension of prime degree, it
follows that there exists an automorphism $\tau\in {\rm Gal}(E'/E)$
such that $\tau\sigma'(\pi)\simeq \pi'\otimes \omega'$. The
automorphism $\tau\sigma'$ of $E'$ extends $\sigma$. Renaming
$\tau\sigma'$ as $\sigma'$, we see that there exists an automorphism
$\sigma'$ of $E'$ extending $\sigma$ such that,
\[ \sigma'(\pi')\simeq \pi'\otimes \omega'.\]
Let $F'$ be the
fixed field of $\sigma'$. Since the fixed field of $\sigma$ is $F$,
$E\cap F'=F$. Further $d$ divides the degree of $E'$ over $F'$ as $\sigma'\vert E=\sigma$, and
$[E':F]=dp$. There are two cases: either $F'=F$ or $F'\neq F$. 

 Suppose $F'=F$. 
Then $E'$ is a cyclic extension of $F$ of
degree $dp$, and $\omega'_F=\omega_F^p$. If $p$ divides $[N:F]$, then
$p^2$ divides $[K':F]=p[N:F]$. It follows that the extension of $F$ defined by
$\omega_F^p$ has a non-trivial intersection with $K'\subset E'$. By 
the inductive hypothesis, the extension of $F$ defined by
$\omega_F'=\omega_F^p$ should be disjoint with $E'$ over $F$. 

This contradiction leads to saying that $p$ is coprime to $[N:F]$. But
then $N$ is contained in the cyclic  extension $K{''}$ of $F$ cut out by
$\omega_F^p$. By the inductive hypothesis, $K{''}\cap E'=F$. It follows that $N=K\cap
E=F$.

Before taking up the case $F'\neq F$, we now rule out the case $F'=F$ by showing that $\pi$ is not cuspidal in this case. Since $\omega_F\circ N_{E/F}$ corresponds to the extension $L$ of $E$ and $E'\subset E$, 
the character $\omega'_F\circ N_{E'/F}=\omega_F^p\circ
N_{E'/F}$ corresponds to the extension $L=KE=K{''}E'$. 
The above process can be continued, and 
the representation  
$\pi'$ (and hence $\pi$) 
is automorphically induced from a cuspidal representation
$\pi_L$ of $GL(n/r, \A_L)$ where $r=[L:E']$. The equation $\sigma'(\pi')\simeq \pi'\otimes \omega'$ implies that $\pi_L$ satisfies
the condition, 
\[ \sigma{''}(\pi_L)\simeq \pi_L\otimes \omega^L,  \]
for some automorphism $\sigma{''}$ of $L$ exending $\sigma'$ on $E'$, and of order equal to $dp$, the order of $\sigma'$. 
Now  $\omega^L_K$ is the trivial character. 
The automorphism
$\sigma{''}^d$ is a non-trivial automorphism of $L$ trivial on $E$, as
$\sigma{''}$ extends $\sigma$ on $E$. But 
\[ \sigma{''}^d(\pi_L)\simeq \pi_L\otimes \omega^L\circ N_{L/K}\simeq
\pi_L.\]
This implies that $\pi$ is not cuspidal, contrary to our hypothesis. 
Hence this rules out the case that $F'=F$.

Hence we are in the situation 
that $F'\neq F$. Since $d$ divides $[E':F']$ and $[E':F]=dp$,  $[E':F']=d$, we see that
$[F':F]=p$ ; moreover we saw that $F'\cap E=F$.  The
extension $E'$ over $F$ is a compositum of the linearly disjoint
extensions $E$ and $F'$ over $F$. The character
$\omega'_{F'}=\omega_F\circ N_{F'/F}$, corresponds to the extension
$KF'$ which contains the compositum $NF'\subset E'$. By the induction
hypothesis,  $KF' \cap E'=F'$. Hence $NF'=F'$, and this implies that $N=K\cap E=F$.

The compositum of the fields $K'$ and $F'$ is contained inside both
$K$ and $E'$. Hence it is disjoint from $E$. If $K'\neq F'$, then the
degree of the compositum $[K'F':F]=p^2$. This contradicts the fact
that the degree of $E'$ over $F$ is $dp$. Hence $F'=K'$. 

We also
obtain that the representation $\pi$ is automorphically induced from a cuspidal
representation $\pi'$ of $GL(n/p. \A_{E'})$, satisfying
\[ \sigma'(\pi')\simeq \pi'\otimes \omega'. \] 
The field $E'$ is also the compositum of the fields $E$ and
$K'$. 
The extension of $E$ defined by the character $\omega'_{F'}\circ
N_{E'/F'}$ is $L$. This can be continued, and we obtain a cuspidal
representation $\pi_L$ of $GL(n/r, \A_L)$ such that 
\[ \pi\simeq AI_L^E(\pi_L),\]
where $r=[L:E]$. The representation $\pi_L$ satisfies, 
\[ \tau(\pi_L)\simeq \pi_L\otimes \omega^L,\]
where $\tau$ is the unique automorphism of $L$ extending $\sigma$ such
that the fixed field of $L$ by $\tau$ is equal to $K$. 

Since $\omega^L_K$ is trivial, 
there exists a Hecke character $\psi$ of $L$ such that
$\tau(\psi)\psi^{-1}=\omega^L$. Then, 
\[ \tau(\pi_L\otimes \psi^{-1})\simeq \pi_L\otimes \omega^L\otimes \tau(\psi)^{-1}
\simeq \pi_L\otimes \psi^{-1}.\]
Hence the cuspidal representation $\pi_L\otimes \psi^{-1}$ is
invariant with respect to the cyclic automorphism group ${\rm
  Gal}(L/K)$ generated by $\tau$. To complete the proof of
Theorem \ref{galoisconjugate},
 we now have to establish descent for cyclic extensions. 

Let $E/F$ be a cyclic extension of degree $d$ and $\pi$ be a cuspidal
representation of $GL_n(\A_E)$ invariant under the action of ${\rm
  Gal}(E/F)=<\sigma>$. Choose an extension $F\subset E'\subset E$ such that
$[E:E']=p$ for some rational prime $p$. By the descent for cyclic
extensions of prime degree, there exists a cuspidal representation
$\pi'$ of  $GL_n(\A_{E'})$ which base changes to $\pi$. We need to
show that $\pi'$ is left invariant by  ${\rm Gal}(E'/F)$. Suppose
$\sigma(\pi')\simeq \pi'\otimes \varepsilon$, where 
$\varepsilon$ is an idele class character of $E'$ which corresponds to
the extension $E/E'$ via class field theory. Since $E/F$ is cyclic,
$\varepsilon=\eta\circ N_{E'/F}$ for some idele class character
$\eta$ of $F$ defining the cyclic extension $E/F$ by class field
theory. Then $\varepsilon_F=\eta^{d/p}$. Since $\eta$ is of order $d$,
$\varepsilon_F$ defines a non-trivial cyclic extension $F{''}$  of degree $p$
of $F$ contained inside $E$. Since $E/E'$ is a cyclic extension of
degree $p$, this implies that $F{''}\subset E'$, thus $E' \cap F'' \neq F$.

This contradicts the first part of
Theorem \ref{galoisconjugate} proved above. Hence $\pi'$ is left
invariant by  ${\rm Gal}(E'/F)$, and by induction can be descended to
a cuspidal representation of $GL_n(\A_F)$.

\section{Proof of Theorem \ref{sbc}}
In this section, we deduce Theorem \ref{sbc} from Theorem
\ref{galoisconjugate}, following the arguments given in \cite{R}. 
\begin{lemma}
 Let $E/F$ be a solvable extension of number fields. Suppose
$\pi$ is a cuspidal automorphic representation of
$GL_n(\A_F)$  such that its base change to $E$ remains cuspidal. 
Let $\chi$ be a non-trivial idele class character on $F$, such that
$\chi\circ N_{E/F}$ is trivial, where $N_{E/F}:C_E\to C_F$ is the norm
map on the idele classes. Then $\pi$ and $\pi\otimes \chi$ are not
isomorphic.  
\end{lemma}
\begin{proof} The hypothesis implies
$\pi\simeq \pi \otimes \chi^k$ for any natural number $k$. Hence we
can assume that $\chi$ is of prime order. In this case, $\chi$ cuts
out a cyclic extension $E'$ of prime degree contained in $F$. By the
characterization of automorphically induced representations
(\cite{AC}),  it
follows that the base change of $\pi$ to $E'$ is not cuspidal. This
contradicts the assumption that the base change of $\pi$ to $E$ is
cuspidal. 
\end{proof}

We first prove Part (2) of Theorem \ref{sbc}, characterizing the
fibres of the base change lift for solvable extensions of number
fields. Suppose $\Pi$ is a 
cuspidal automorphic representation of
$GL_n(\A_E)$. Let $\pi, ~\pi'$ be 
cuspidal automorphic representations of
$GL_n(\A_F)$ which base change to $\Pi$. We need to show that
$\pi'\simeq \pi\otimes \chi$ for some Hecke character $\chi$ on $F$
such that $\chi\circ N_{E/F}=1$. 

By the results of  \cite{AC}, the theorem is true for 
cyclic extensions of prime degree. Let $E\supset E_1\supset F$ be a
tower of Galois extensions of $F$,  where $E_1/F$ is of prime
degree. Let $BC_{E_1/F}(\pi)=\pi_1$ and $BC_{E_1/F}(\pi')=\pi'_1$.
By induction, we assume that the theorem is true for the extension
$E/E_1$. We have, $\pi_1'\simeq \pi_1\otimes \chi_1$ for some Hecke
character $\chi_1$ on $E_1$ such that $\chi_1\circ N_{E/E_1}=1$.
 Let $\sigma$ be a generator of ${\rm Gal} (E_1/F)$. We have
\[ \pi_1\otimes \chi_1\simeq \pi_1'\simeq {^{\sigma} \pi'_1} \simeq
{^{\sigma}\pi'_1}\otimes {^{\sigma} \chi_1}\simeq \pi_1\otimes
{^{\sigma}\chi_1}.\]
\[{\rm Hence}\,\,\,\,\,\,\, \pi_1\simeq \pi_1\otimes{^{1-\sigma}
  \chi_1}.\]
If $\chi_1\neq {^{\sigma}\chi_1}$, let $f$ denote the order of
$^{1-\sigma}\chi_1$, $p$ a prime dividing $f$, and let 
\[\nu={^{(1-\sigma)f/p}\chi_1}.\]
$\nu$ is a non-trivial character of ${\rm Gal}(E/E_1)$ of order $p$
satisfying,
\[\pi_1\simeq \pi_1\otimes \nu.\]
It follows  from the characterization of automorphic induction, 
that  $\pi_1$ is automorphically induced from a cuspidal representation
$\pi_{\nu}$ belonging to the class field $E_{\nu}$ defined by $\nu$. But
$E_{\nu}\subset E$, and it follows that $\Pi$ is not cuspidal,
contrary to our assumption on $\Pi$. Hence we have that $\chi_1$ is
invariant by ${\rm Gal}(E_1/F)$ and descends to an idele class character 
$\chi_2$ of $C_F$  such that 
$\chi_1=\chi_2\circ N_{K_1/K_2}$. Then 
\[ BC_{E_1/F}(\pi_2\otimes \chi_2)\simeq \pi_1\otimes \chi_1\simeq
\pi_1'\simeq BC_{E_1/F}(\pi_2').\]
Hence we have a Hecke character $\theta$ corresponding to a character
of ${\rm Gal}(E_1/F)$, such that
\[\pi_2'\simeq \pi_2\otimes \chi_2\theta,\]
and $\chi_2\theta$ defines a character of ${\rm Gal}(E/F)$. This proves Part
(2) of Theorem \ref{sbc}, as the distinction between $\pi$ and
$\pi\otimes \chi$ follows from the properties of automorphic
induction. 

We now move on to proving Part (1) of Theorem \ref{sbc}.  We prove a
preliminary lemma, which also proves the uniqueness assertion about
$\psi$ in Theorem \ref{sbc}.  

\begin{lemma}\label{char}
Let $E/F$ be a solvable extension, and $\Pi$ be a 
cuspidal automorphic representation of $GL_n(\mathbf{A}_K)$. Suppose
$\chi$ is a ${\rm Gal}(E/F)$ invariant idele class  character of $E$,  
such that both $\Pi$ and $\Pi\otimes \chi$
are in the image of base change from $F$. Then $\chi$ lies in the
image of base change. 
\end{lemma}
\begin{proof} The proof is by induction, and is true for cyclic extensions
of prime degree. Assume we have $E\supset E_1 \supset
F$ with $E/E_1$ prime of degree $p$. 
Since $\chi$ is invariant, $\chi=\chi_1\circ N_{E/E_1}$ for some idele class
character $\chi_1$ of $C_{E_1}$. Suppose $\pi_1,~\pi_1'$ are cuspidal
automorphic representations of  $GL_n(\mathbf{A}_{E_1})$ which base
change respectively to $\Pi$ and $\Pi\otimes \chi$. 
Since both $\pi_1'$ and $\pi_1\otimes \chi_1$ base change to $\Pi\otimes
\chi$, by the description of the fibres of base change for a cyclic
extension  of prime degree, we obtain,
\[ \pi_1'\simeq \pi_1\otimes \chi_1\eta_1,\]
for some Hecke character $\eta_1$ of $C_{E_1}$ vanishing on $N_{E/E_1}C_E$.

Assume further, as we may from the hypothesis, that both $\pi_1$ and $\pi_1'$ lie in the
image of base change from $F$ to $E_1$.  For any $\sigma \in
{\rm Gal}(E_1/F)$, 
\[ \pi_1\otimes \chi_1\eta_1\simeq \pi_1'\simeq {\sigma}(\pi_1')\simeq 
\pi_1\otimes \sigma(\chi_1\eta_1).\]
Hence, $\pi_1\simeq \pi_1\otimes \nu$, where 
$\nu=\sigma(\chi_1\eta_1)(\chi_1\eta_1)^{-1}$. 
Since $\chi$ is ${\rm Gal}(E/F)$-invariant, we have 
$^{\sigma} \chi_1=\chi_1\varepsilon_1^i$ and $^{\sigma} \eta_1=\varepsilon_1^j$
for some integers $i, ~j$, $\varepsilon_1$ being a character associated to $E/E_1$. Hence $\nu=\varepsilon_1^l$
for some integer $l$. Since $\Pi$ is cuspidal, the cuspidality
criterion for automorphic induction implies that $\nu=1$. Hence we get
that $\chi_1\eta_1$ is invariant by ${\rm Gal}(E_1/F)$. By induction,
$\chi_1\eta_1$ lies in the image of base change from $F$ to $E_1$,
and it follows that $\chi$ also lies  in the image of base change
from $F$ to $E$. 
\end{proof}

With this lemma, we now proceed to the proof of Part (1) of Theorem
\ref{sbc}.  The proof is by induction on the degree of the extension
$E$ over $F$. By the results of \cite{AC}, it is true for extensions
of prime degree. We now assume  there is a sequence of fields 
 \[E\supset E' \supset F,\]
where $E'/F$ is a cyclic extension of prime degree $p$. By the
inductive hypothesis,  there exists a ${\rm Gal} (E/E')$-invariant
idele class character 
$\psi_0$ of $E$, and a cuspidal automorphic representation $\pi'$ 
of $GL(n,\mathbf{A}_{E'})$ such that 
\[\Pi\otimes \psi_0=BC_{E/E'}(\pi'). \]
Let $\tau'$ be a
generator of ${\rm Gal} (E'/F)$, and let $\tau$ be an element of
${\rm Gal}(E/F)$ lifting $\tau'$. Then 
\[ BC_{E/E'}(\tau'(\pi'))\simeq~ ^{\tau}\Pi \otimes
\tau(\psi_0)\simeq \Pi\otimes \tau(\psi_0)\simeq  (\Pi\otimes
\psi')\otimes \tau(\psi_0)\psi_0^{-1}.\]
Since ${\rm Gal}(E/E')$ is a normal subgroup of ${\rm Gal}(E/F)$, for
any $\sigma \in  {\rm Gal}(E/E')$,  
\[ \sigma(\tau(\psi_0))=\tau(\tau^{-1}\sigma\tau)(\psi_0)=\tau(\psi_0).\]
Hence $ \tau(\psi_0)\psi_0^{-1}$ is ${\rm Gal}(E/E')$-invariant. 
Since both $ \Pi\otimes \psi'$ and 
$ (\Pi\otimes \psi_0)\otimes \psi_0^{-1+\tau}$ lie in the image of
base change from $E'$ to $E$, by Lemma \ref{char}
there exists an idele class character $\chi'$ of $E'$,  such that
$\tau(\psi_0)\psi_0^{-1}=\chi'\circ N_{E/E'}$. Hence, 
\[BC_{E/E'}(\tau'(\pi')\otimes \chi'^{-1})\simeq
  \Pi\otimes \psi_0.\]
By Part (2) of Theorem \ref{sbc}, characterizing the fibres of the base change
lift,  we conclude that there is an idele class character $\chi{''}$
corresponding via class field theory to a character of  ${\rm Gal}(E/E')$,
such that  
\begin{equation}\label{crucial}
 {\tau'}(\pi') \simeq \pi\otimes \chi{'}\chi''=\pi'\otimes \eta',
\end{equation}
where  $\eta'=\chi'\chi''$. Further, 
\begin{equation} \label{eta'}
\eta'\circ N_{E/E'}=\chi'\circ N_{E/E'}=\tau(\psi_0)\psi_0^{-1}.
\end{equation}
Write the elements of 
\[{\rm Gal}(E/F)=\{\tau^{-i}\sigma\mid 0\leq i
<p, ~\sigma\in {\rm Gal}(E/E').\]
We have for $x\in C_E$, 
\begin{equation*}
\begin{split}
\eta'_F\circ N_{E/F}(x) &= \eta'\left(\prod_{i=0}^{p-1}\prod_{\sigma \in {\rm
    Gal}(E/E')}\tau^{-i}\sigma x\right)
= \prod_{i=0}^{p-1}\tau^i\eta'(N_{E/E'}(x))\\
 &=  \prod_{i=0}^{p-1}\tau^i(\tau(\psi_0)\psi_0^{-1})\\
&=\tau^p(\psi_0)\psi_0^{-1}.
\end{split}
\end{equation*}
Since $\psi_0$ is ${\rm Gal}(E/E')$-invariant and $\tau^p\in {\rm
  Gal}(E/E')$,  it follows that $\eta'_F\circ N_{E/F}$ is
trivial. Hence by Part
(1) of Theorem \ref{galoisconjugate}, and Equation (\ref{crucial}),
$\eta'_F$ is trivial.  

Let $\alpha$ be an idele class character of $E'$, satisfying 
$\alpha\tau'(\alpha)^{-1}=\eta'.$ By equation (\ref{crucial}), 
\[\tau'(\pi'\otimes \alpha) =\tau'(\pi')\otimes
 \alpha\eta'^{-1} = \pi'\otimes \alpha.\]
Hence $\pi'\otimes \alpha$ is ${\rm Gal}(E'/F)$-invariant, and 
descends to $F$.  
Hence we obtain that $\Pi\otimes \psi_0\otimes (\alpha\circ N_{E/E'})$ descends. 

To finish the proof, we have to check that
$\psi_0\otimes (\alpha\circ N_{E/E'})$ is ${\rm Gal}(E/F)$-invariant. 
For this it is enough to check that $\psi_0 \otimes 
(\alpha\circ N_{E/E'})$ is $\tau$-invariant:
\begin{equation*}
\begin{split}
\tau(\psi_0\otimes (\alpha\circ N_{E/E'})) 
&= \tau(\psi_0)\otimes \tau(\alpha)\circ N_{E/E'}\\
& =\tau(\psi_0)\otimes (\alpha\circ N_{E/E'})(\eta'\circ N_{E/E'})^{-1}\\ 
&= \psi_0 \otimes (\alpha \circ N_{E/E'}), 
\end{split}
\end{equation*}
where the last equality follows from equation (\ref{eta'}). 

\section{Trace Formula}
We want to prove Theorem \ref{nonexistence} 
ruling out the existence of a cuspidal
representation $\Pi$ of $GL(n, \A_E)$ satisfying Equation
\ref{lrconj},  
\[
^{\sigma}\Pi \simeq \Pi \otimes \omega,  
\]
where $\omega$ is an idele class character of $E$ such that its
restriction to $C_F$ corresponds by Artin reciprocity to a primitive
character of the cyclic group ${\rm Gal}(E/F)$. 

If $\Pi$ satisfies Equation \ref{lrconj}, it will, for a suitable choice of a
function $\phi\in C^{\infty}_c(GL(n, \A_E))$, contribute a non-zero
term to the trace, 
\begin{equation}\label{twistdisctrace}
\mbox{Trace}(I_{\theta}(R_{disc}\otimes \omega^{-1})(\phi)).
\end{equation}
Here $R_{disc}$ is the discrete part of the representation of $GL(n,
\A_E)$ on 
\[{\mathcal A}_2:=L^2(GL(n,E)A\backslash GL(n, \A_E)).\]
Here $A$ is $\R^\times_+$ embedded diagonally into $GL(n, E_w)$ at all
archimedean places $w$ of $E$. The operator $I_{\theta}$ is given by
$\phi(g)\mapsto \phi(\sigma^{-1}(g))$, where $\sigma$ is our chosen
generator of $\Sigma=\mbox{Gal}(E/F)$. 

There is a general formula for the trace in Equation
(\ref{twistdisctrace}) due to Kottwitz-Shelstad (\cite{KS}) and
Moeglin-Waldspurger (\cite{MWII}). (In fact this trace must be completed
by Arthur's ``discrete terms'', which we will describe presently in
our case). The formula is, 
\begin{equation}\label{tf2.1}
T_{disc}(\phi \times \theta ; \omega^{-1})
=\sum_{\G'\subset \cE} \iota(\G')ST_{disc}^{G'}(\phi^{G'}).
\end{equation}
Here $\G'$ runs over the elliptic endoscopic data consisting of
triples of the form $(G', \cG', \tilde{s})$: these will be reviewed in
the next paragraph; $G'$ is a reductive $F$-group and $\phi^{G'}$ is a
function on $G'(\A_F)$ associated to $\phi$. 

The so-called 'stable discrete trace formula' $ST^{G'}_{disc}$ will be
very simple in our case, as $G'$ will be a group $GL(m)$: see paragraph 7. The terms in
\[T_{disc}(\phi \times \theta ; \omega^{-1}) \] 
are as follows: 

\begin{enumerate}
\item 
The traces $\mbox{Trace}(I_{\theta}(\Pi \otimes \omega^{-1})(\phi))$ for a
cuspidal representation $\Pi$ of $GL(n, \A_E)$ such that
$^{\sigma}\Pi\simeq \Pi \otimes \omega$. The operator $I_{\theta}$
sends an automorphic form $f(g)\mapsto f(\sigma^{-1}(g)), ~g\in GL(n,
\A_E)$. The cusp forms occur with multiplicity one and $I_{\theta}$
is an intertwining operator sending $\Pi$ to $^{\sigma}\Pi$. 

\item Similar traces where $\Pi$ belongs to the discrete spectrum (and
  is not cuspidal) (see \cite{MWI}). This means that $n=ab$, and that there exists
  $\pi_a$, a cuspidal representation of $GL(a, \A_E)$ such that $\Pi$ is
  a quotient of the representation 
\[\rho=\pi_a|.|^{\frac{b-1}{2}}\boxplus \pi_a|.|^{\frac{b-3}{2}}\boxplus\cdots
\boxplus \pi_a|.|^{\frac{1-b}{2}},\]
where $|.|$ denotes the idele norm, seen as a character of  $GL(a,
\A_E)$ via the determinant; and the notation $\boxplus$ denotes, as
usual, parabolic induction, here from the parabolic subgroup of
$GL(n)$ of type $(a, \cdots, a)$. 

Now if $^{\sigma}\Pi\simeq \Pi \otimes \omega$, the same is true of
$\rho$. Since the representation $\pi_a$ is almost tempered, this
implies that $^{\sigma}\pi_a\simeq \pi_a \otimes \omega$. By induction
(since $a<n$), this is impossible. 

\item There are now the discrete terms defined by Arthur, which do not
  come from the discrete spectrum. We first consider the simplest
  case. Here $\Pi=\pi_1\boxplus \cdots \boxplus \pi_t, $ where $\pi_i$
  is a cuspidal representation of $GL(n_i, \A_E)$, and
 $ ~\sum_{i=1}^tn_i=n$. We assume then $^{\sigma}\Pi\simeq \Pi \otimes
  \omega$; of course
\[ ^{\sigma}\Pi= ^{\sigma}\pi_1\boxplus \cdots \boxplus
^{\sigma}\pi_t, \]
and this equivalence implies that there is an element $s\in W_M$, the
Weyl group corresponding to the Levi component $GL(n_1)\times \cdots
GL(n_t)$, such that 
\[ ^{\sigma}\pi_i\simeq \pi_{s(i)}\omega.\]
We must further assume that $s$ is `regular', i.e., ${\mathfrak
  a}_M^s={\mathfrak a}_G$, where ${\mathfrak a}_M $ (resp. ${\mathfrak
  a}_G$) denote the split component of the centers of $M$ and $G$
respectively. This implies that $M$ is homogeneous ($n=ab$) and that 
$^{\sigma}\pi_i\simeq \pi_{s(i)}\omega$, where $s$ is a cyclic
permutation of order $b$. The corresponding term is the trace of the product of $\Pi \omega^{-1}(\phi)$,  and of an intertwining operator associated to $s \times \sigma$, defined by Arthur, acting on the space of $\Pi$. Its precise form will be irrelevant.

\item Finally, we can build similar terms with $\Pi_i$ cuspidal
  replaced by a residual, discrete spectrum representations as in (2)
  above.  
\end{enumerate}

We note that all the representations of type (1,2) occur with multiplicity 1. Furthermore their Hecke eigenvalues are independent from those of the representations of type (3,4).

In the next paragraph, we compute the right-hand side, i.e., the endoscopic
terms.

For more information on the endoscopic stabilisation of the trace
formula, and in particular the use of formula \ref{tf2.1}, we refer to 
\cite{MWII}; in particular sections I.6.4 and X.5.9. Suffice it to say
here that if $\Pi$ is a representation of $GL(n, \A_E)$ occurring in
the left-hand side of (\ref{tf2.1}), i.e., in the discrete part of
Arthur's trace formula as reviewed above, there will be an endoscopic
group $G'$, and a representation $\pi'$ of $G'(\A_F)$ such that $\pi'$
and $\Pi$ are associated, i.e., the Hecke matrices of $\Pi$ are
deduced at almost all primes from those of $\pi'$ in a prescribed
manner, determined by the endoscopic datum, given in \cite[Section 6.4]{MWI}.
In our case, there will be a unique datum $\G'$ (or none at all) and
the relation between $\pi'$ and $\Pi$ will be quite explicit. 

\section{Endoscopic data} 

We now consider the right-hand side of Equation (\ref{tf2.1}). We must first
describe the endoscopic data. We use Waldspurger's formalism for base
change (\cite{W}, \cite{BC}). 

We consider $GL(n)/E$ as an $F$-group by restriction of scalars and
denote it by $G$. We will sometimes denote by $G_0$ the group $GL(n)$
over $E$. The generator $\sigma$ of $\Sigma={\rm Gal}(E/F)$ acts on
$G$ by $F$-automorphisms; as such we denote it by $\theta$. We fix an
isomorphism $\sigma\mapsto \iota(\sigma)$ between $\Sigma$ and
$\Z/d\Z$. For $w \in W_F, \iota(w)$ is then defined by composition. 

The connected component of identity of the dual group of $G$ is $\hat{G}=GL(n, \C)^d=\prod_{i\in
  \Z/d\Z}GL(n,\C)$; the $F$-structure on $G$ gives an action of ${\rm
  Gal}(\bar{F}/F)$ on $\hat{G}$ quotienting through $\Sigma$:
\begin{equation}\label{sigmaonhatG}
\sigma(g_1, \cdots, g_d)=(g_{1+\iota(\sigma)},
\cdots,g_{d+\iota(\sigma)}).
\end{equation}
Then ${^LG}=\hat{G}\rtimes W_F$, the action of $W_F$ being so obtained. On
the other hand, $\theta$ defines an automorphism $\hat{\theta}$ of
$\hat{G}$, 
\[ \hat{\theta}(g_1, \cdots, g_d)=(g_2, \cdots, g_d, g_1).\]
We are given a character $\omega$ of $\A_E^\times$, which defines via the
determinant an abelian character of $G(\A_F)=GL(n, \A_E)$. By a result
of Langlands \cite{L3} we can see $\omega$ as an element  $\bf{a}$ $ \in
H^1(W_F, Z(\hat{G}))$. Note that $Z(\hat{G})=(\C^\times)^d$, the action of
$W_F$ being given by equation (\ref{sigmaonhatG}). In general, the
element $\bf{a}$ is only defined modulo the group, 
\[ {\rm ker}^1(F, Z(\hat{G}))={\rm ker}\left( H^1(W_F, Z(\hat{G}))\to
  \bigoplus_v  H^1(W_{F_v}, Z(\hat{G}))\right),\]
where $v$ ranges over the places of $F$ (see \cite{KS}, \cite{W}). 

In our case, however, Shapiro's lemma implies that 
\[ H^1(W_F, Z(\hat{G}))=H^1(W_E, \C^\times),\]
with trivial action of $W_E$. Hence, 
\[ H^1(W_F, Z(\hat{G}))={\rm Hom}_{ct}(W_E, \C^\times)={\rm Hom}_{ct}(C_E,
\C^\times),\]
where $C_E$ is the group of idele classes. Similarly, for a place $v$
of $F$, 
\[ H^1(W_{F_v}, Z(\hat{G}))=\bigoplus_{w|v} H^1(W_{E_w}, \C^\times).\]
Thus, ${\rm ker}^1(F, Z(\hat{G}))$ is the group of idele class
characters that are locally trivial, so 
\[{\rm ker}^1(F, Z(\hat{G}))=\{1\}.\]
Now an endoscopic datum for $(G, \theta, \bf{a})$ is a triple, 
${\bf G}'=(G', {\mathcal G}',\tilde{s}=s\tilde{\theta})$ subject to the
following conditions:

\begin{description}

\item[(E1)] $G'$ is a quasisplit connected reductive group over $F$. 

\item[(E2)] $\tilde{s}=s\tilde{\theta}$ is a semisimple element in
  $\hat{G}\rtimes \Theta$, where $\Theta=<\hat{\theta}>\simeq
  \Z/d\Z$. 

\item[(E3)] ${\mathcal G}'\subset {^LG}$ is a closed subgroup. 

\item[(E4)] There exists a split exact sequence, 
\[ 1\to \hat{G}_{\tilde{s}}\to {\mathcal G}'\to W_F\to 1,\]
where $ \hat{G}_{\tilde{s}}$ is the connected component of the
centralizer of $\tilde{s}$ and ${\mathcal G}'\to W_F$ is induced by
the map ${^LG}\to W_F$. In particular, ${\mathcal G}'\cap
\hat{G}=\hat{G}_{\tilde{s}}$. 

\item[(E5)] For $(g,w)\in {\mathcal G}' $, 
\[ s\hat{\theta}(g)w(s)^{-1}=a(w)g,\]
where $a(w)$ is a $1$-cocycle of $W_F$ with values in $Z(\hat{G})$ and
defining $\bf{a}$. 

\end{description}

We note that any semisimple $\tilde{s}=s\hat{\theta}$ is conjugate to
an element $\tilde{s}$ such that $s=(s_0, 1,\cdots, 1)$. In this case,
$\hat{H}= \hat{G}_{\tilde{s}}=\hat{G}_{0, s_0}$ is diagonally embedded in
$\hat{G}$. Here  $\hat{G}_{0, s_0}$ is the centralizer of $s_0$ in
$\hat{G}_0$, which is connected. Thus, 
\[\begin{split}
\hat{G}_{\tilde{s}} & =\{(h, \cdots, h)\mid h \in \hat{H}\}\\
&:=\{{\rm diag}(h)\mid h \in \hat{H}\},
\end{split}
\]
where ${\rm diag}: GL(n,\C) \to GL(n, \C)^d$ is the diagonal map. We
look for 
\begin{equation}\label{xi}
 \xi: \cG'\to {^LG}
\end{equation}
where $\cG'$ admits an exact sequence $ (E4)$. Thus for $h\in
\hat{H}$, 
\[ \xi: (h, 1)\mapsto ({\rm diag}(h), 1),\]
while for $w\in W_F$, 
\[\xi:  (1,w)\mapsto (n(w), w)=(n(w), 1) (1,w).\]
Here we have chosen a splitting $n:W_F \to \cG'$ for $\cG'$. Let us
denote by $h\mapsto {^wh}$ the action of $W_F$ on $\hat{H}$ coming
from ${ (E4)}$. Then
\[ ({^wh}, 1)=(n, w)(h,1)(n, w)^{-1},\]
where we are writing for short  $n=n(w)$ and $(h, 1)=({\rm diag}(h),
1)$. Hence
\[\begin{split}
 ({^wh}, 1)& = (n,1)(1,w)(h,1)(1,w)^{-1}(n,1)^{-1}\\
&=(n,1)(h,1)(n,1)^{-1},
\end{split}
\]
since $h$, being diagonal in $\hat{G}$, is invariant by the action   
(\ref{sigmaonhatG}) of $W_F$. Write $ n=n(w)=(n_1,\cdots, n_d)$, so 
\begin{equation}\label{1}
 n \,{\rm diag}(h) n^{-1}=(n_1hn_1^{-1}, \cdots, n_dhn_d^{-1})={\rm
    diag}(h'),
\end{equation}
for some $h'\in \hat{H}$. 

We now assume that $s_0=(s_1, \cdots, s_d)$ is given by diagonal
scalar matrices $s_i$ of degree $b_i$ with distinct eigenvalues
$t_i$. Then
\[ \hat{H}=\prod_{i=1}^a GL(b_i)\subset GL(n).\]
Write $a=a_1+\cdots+a_r, ~(a_k\geq 1)$, with 
\[ b_1=b_2=\cdots =b_{a_1}< b_{a_1+1}=\cdots=b_{a_2}<\cdots\]
Since $n_i$ normalizes $\hat{H}$, 
\[ n_i\in \prod_{k=1}^r GL(b_{k})^{a_{k}}\rtimes \fS_{a_{k}}\]
with obvious notation.
We choose explicitly as representatives of the Weyl group
$\fS_{a_{k}}$ the obvious block matrices with blocks of size
$b_{k}$ equal to  either $0$ or $1$. 

Write $W=\prod_{k=1}^r\fS_{a_k}$, so that the normaliser of $\hat{H}$
is $\hat{H}W$. By Equation (\ref{1}), 
\[{\rm Ad}(n_i)h\equiv {\rm Ad}(n_j)h \quad \forall i, j.\]
Hence $n_i= h_i\tau$, where $\tau\in W$ is independent of $i$ and
$h_i\in \hat{H}$; morever
\[{\rm Ad}(h_i)h\equiv {\rm Ad}(h_j)h \quad \forall i, j.\]
Thus $h_i=z_{ij}h_j$ with $z_{ij}\in Z(\hat{H})=(\C^\times)^r$. Hence we
can write
\begin{equation}\label{2}
n(w)=(z_i(w)h(w)\tau(w))_i, \quad\mbox{where} \quad z_i(w)=(z_{i1}(w),
\cdots, z_{ir}(w))\in (\C^\times)^r.
\end{equation}
In the stabilisation of the trace formula, we are only interested in
the elliptic endoscopic data, i.e., those such that the neutral
component of $Z(\hat{H})^{W_F}$ and of  $Z(\hat{H})^{W_F,
  \hat{\theta}}$ coincide. The second group is equal to $\C^\times$
embedded diagonally in $GL(n, \C)^d$. We have
$Z(\hat{H})=\prod_{k=1}^r(\C^\times)^{a_k}$ and $n(w)$ acts by $\tau(w)\in
\prod_k\fS_{a_k}$. Thus  $Z(\hat{H})^{W_F}$ is the set of fixed points of
the $\tau(w), ~w \in W_F$. In particular, it contains the product
$\prod_k\C^\times$, embedded diagonally in $\prod_kGL(b_k)$. 

If $H$ is elliptic, we see that $r=1$, so $\hat{H}=GL(b)^k$ is
homogeneous. Furthermore, $W_F$ acts on $(\C^\times)^a$ via $\tau(w)\in
\fS_a$. The image of $W_F$ by $w\mapsto \tau(w)$ must therefore be a
transitive subgroup of $\fS_a$. 

So far we have shown that $\hat{H}=GL(b)^k$, and 
\begin{equation}\label{3}
n(w)=(n_i(w)), \quad\mbox{where} \quad n_i(w)=z_i(w)h(w)\tau(w)  
\end{equation}
with $z_i(w)\in Z(\hat{H})\simeq  (\C^\times)^r$.

The group $\cG'=\hat{H}\rtimes W_F$ is defined as a semi-direct
product, by the conjugation action of $n(w)$ on $\hat{H}$. Dually, 
$H\times_F \bar{F}\simeq GL(b)^a/\bar{F}$, where the rational
structure will be described presently. In particular, the derived
subgroup of $H$ is simply connected. This implies (see \cite[Section
2.2]{KS}), that $\cH'$ is an $L$-group, i.e., that for a suitable
choice of section the action of $W_F$ on $\hat{H}$ preserves a Borel
subgroup and a splitting. 

We have seen that $n(w)=z_i(w)h(w)\tau(w)$ acts by conjugation on
$\hat{H}$. If $h(w)=1$, this is easily seen to preserve a
splitting. Conversely, if $n(w)$ preserves a splitting, one checks
that the $h(w)\in \hat{H}$ must act trivially by conjugation, so we may
assume $h(w)\equiv 1$. With this section (if it is one), $\cH'\cong
{^LH}$ is naturally embedded in ${^LG}$, whence a homomorphism of
$L$-groups, 
\[ \xi_1: {^LH}\to {^LG}.\]
The contribution of this endoscopic datum will be deduced from
$\xi_1$. 

Since now $n(w)=(z_i(w)\tau(w))_i$, we must still check the cocycle
relation
\[n(ww')=n(w).wn(w'),\]
where the action of $w$ is given by the structure of ${^LG}$. If
$w\in W_F$ is sent to $\sigma^k\in \Sigma$, with $k=\iota(w)$, this
says that 
\[n_i(ww')=n_i(w)n_{i+k}(w').\]
Write $z_i(w)=(z_{i, \alpha}(w))$ according to the decomposition
$Z(\hat{H})=(\C^\times)^r$ for $\alpha=1,\cdots, a$. Thus
\[ \begin{split}
z_i(ww')\tau(ww') &=z_i(w)\tau(w)z_{i+k}(w)\tau(w')\\
\mbox{i.e.}\quad z_i(ww') &= z_i(w).\tau z_{i+k}(w')\tau^{-1},
\end{split}
\]
with $\tau=\tau(w)\in \fS_a$. Now
$\tau((z_{\alpha}))=z_{\tau^{-1}\alpha}$, so the cocycle relation
reads: 
\begin{equation}\label{4}
z_{i, \alpha}(ww')=z_{i, \alpha}(w)z_{i+k,\tau^{-1} \alpha}(ww'),
\end{equation}
where $k=\iota(w)$. 

\section{ endoscopy, with character}
We now have to introduce the character $\omega$ in the endoscopic
computations. This intervenes through formula ${(E 5)}$ in the
definition of endoscopic datum. We want to make the element 
$\bf{a}$ $ \in H^1(W_F, Z(\hat{G}))$, or rather a representative $a\in 
 Z^1(W_F, Z(\hat{G}))$, explicit. We write for $w\in W_F$: 
\[ a(w)=(a_i(w)), \quad a_i(w)\in \C^\times.\] Since $\omega$ is a character of $C_E$, it can be identified with an element of $H^1(W_E, \C^\times)$. We now need Shapiro's lemma.
For its explicit description, we follow Langlands
\cite{L1} (see also Serre  \cite{S}). Recall that 
\[W_E\backslash W_F \simeq {\rm Gal}(E/F)\simeq \Z/d\Z,\]
the isomorphism sending the generator $\sigma$ of $\Sigma$ to $1$. We
choose a representative $\sigma\in W_F$ of this generator, which we
also denote by $\sigma$. Now $\{\sigma, \sigma^2,\cdots, \sigma^d\}$ are
representatives of $ W_E\backslash W_F$. Note that $\sigma^d\neq 1$ as follows from class field theory, cf. (7.5), (7.6) below.

For any $w\in W_F$, 
\[\sigma^i(w)= \delta_i(w)\sigma^j,\]
where $j\equiv i+k ({\rm mod}  d) \quad \mbox{if} \quad k=\iota(w)$ 
(see Equation (\ref{sigmaonhatG})) and $\delta_i(w)\in W_E$. We set 
$a_i(w)=\omega(\delta_i(w))$. For $w\in W_E$,
$\sigma^iw=\delta_i(w)\sigma^i$, so 
\begin{equation}\label{5}
a_i(w)=(\omega(\sigma^iw\sigma^{-i}))=\omega({^{\sigma^i}w}):=\omega_i(w), 
\end{equation}
since the lifting $\sigma$ acts by conjugation, on the abelianized
Weil group $C_E$ of $E$, through its image in $Gal(E/F)$. 

Consider our chosen lift $\sigma\in W_F$. Then 
\begin{equation}\label{6}
\sigma^i\sigma=\sigma^{i+1}=\delta_i(\sigma)\sigma^{[i+1]},
\end{equation}
where $[i+1]$ is the representative of $i+1$ in $\{1,\cdots, d\}$. The
foregoing equation implies
\begin{equation}\label{7}
\delta_i(\sigma) =1\quad\mbox{for}\quad (i=1,\cdots, d-1), \quad
\delta_d(\sigma)=\sigma^d\in W_E.
\end{equation}
This defines completely $a(w)$. Now, 
\[ \phi: W_F\to Z(\hat{G})\rtimes W_F, \quad w\mapsto (a_i(w), w),\]
defines an $L$-parameter for the $L$-group of ${\rm Res}_{E/F}GL(1)$,
corresponding to the character $\omega$. 

We now have to introduce the condition
\begin{equation}\label{8}
\omega|_{\A_F^\times}=\varepsilon_{E/F}.
\end{equation}
In cohomological terms, this is given by the corestriction, 
\[ {\rm Cor} : H^1(W_E, \C^\times)\to H^1(W_F, \C^\times),\]
dual to the transfer map $W_F/W_F^{\rm der}\to W_E/W_E^{\rm der}$. 
Explicitly, this is given as in (\cite[Chapter VII, Section 8]{S}) by
\[ w\mapsto \prod_i\delta_i(w).\]
Our condition is therefore, for $w\in W_F$:
\begin{equation}\label{9}
\prod_ia_i(w)=\prod_i\omega(\delta_i(w))=\varepsilon_{E/F}(w),
\end{equation}
where both sides are seen as characters of the Weil groups (recall
that $\varepsilon_{E/F}$ is here seen as a character of the Galois
group). For $w\in W_E$, $\varepsilon_{E/F}(w)=1$ and the left hand side
is 
\[ \prod_i{^{\sigma^i}\!\omega}(w)=\omega(N_{E/F}z),\]
where $z\in C_E$ is the image of $w$. Since $\omega$ restricts trivially to
$N_{E/F}(C_E)\subset C_F$, the relation is satisfied. 

On the other hand, $(\delta_i(\sigma))=(1,\cdots, 1, \sigma^d)$, thus 
equation (\ref{9}) is equivalent to
\begin{equation}\label{10}
a(\sigma)=(1, \cdots, \zeta),
\end{equation}
where $\zeta=\varepsilon_{E/F}(\sigma)$ is a primitive root of unity of
order $d$. (In particular, $\omega (\sigma^d)= \zeta$ when $\sigma^d$ is seen as an element of $W_{E/F}$, hence of $C_E$...) 

Now consider the condition $({E5})$ on the endoscopic group:
\begin{equation}\label{11}
s\hat{\theta}(g)w(s)^{-1}=a(w)g, \quad (g, w)\in \cG'.
\end{equation}
For $w=1$, this is the condition defining $\hat{H}=\hat{G}'$. Consider
the image $(n(w), w)$ of the section $n$ for $w\in W_F$. We have
\[ n(w)=(z_i(w)\tau(w))_i.\]
Write for simplicity $\tau, z_i, a_i$ for $\tau(w), z_i(w), a_i(w)$
respectively. If $k=\iota(w)$, 
\[ w(s)=w(s_0, 1, \cdots, 1)=(1, \cdots, s_0, \cdots, 1),\]
where $s_0$ occurs at the place $l=d+1-k$ with the convention that
$l=1$ if $k=0$. The equation (\ref{11}) reads, 
\begin{equation}\label{12}
(s_0, 1,\cdots,1)(z_2\tau, \cdots, z_1\tau)(1,\cdots, 1,
s_0^{-1},1)=(a_i)(z_i\tau).
 \end{equation}
Thus, 
\begin{equation}\label{13}
s_0z_2\tau=a_1z_1\tau,\quad z_3\tau=a_2z_2\tau, \quad \cdots 
\quad z_{l+1}\tau s_0^{-1} =a_lz_l\tau, \quad \cdots
\quad z_1\tau=a_dz_d\tau.
\end{equation}
(If $k=0, ~s_0z_2\tau s_0^{-1}=a_1z_1\tau$ etc.; if $k=1, ~l=d$ and
the last equation is $z_1\tau s_0^{-1}=a_dz_d\tau$.) 
Write ${^{\tau}\!s_o^{-1}}=\tau s_0^{-1} \tau^{-1}$. Then equation
(\ref{13}) is equivalent to, 
\begin{equation}\label{14}
\begin{split}
s_0z_2& =a_1z_1\\
z_3&=a_2z_2\\
&.\\
&.\\
z_{l+1}{^{\tau}\! s_0^{-1}}& =a_lz_l\\
&.\\
&.\\
z_1&=a_dz_d.
\end{split}
\end{equation}
Note that all these elements are contained in $Z(\hat{H})$, hence
commute. (For $k=0$, the first line is
$s_0z_0\,{^{\tau}\!s_0^{-1}}=a_1z_1$; for $k=1$, the last line is
$z_1\,{^{\tau}\!s_0^{-1}}=a_dz_d)$. Taking the product, we see that 
\[ s_0\,{^{\tau}\!s_0^{-1}}=\prod a_i=\varepsilon_{E/F}(w),\]
by equation (\ref{9}), so
$s_0=\varepsilon_{E/F}(w)\,{^{\tau}\!s_0}$. 

Now $s_0\in Z(\hat{H})=(\C^\times)^a$, and $\tau\in \fS_a$. Write
$s_0=(s_{0, \alpha}), ~\alpha=1, \cdots, a$. Thus $^{\tau}\!s_0=
(s_{0, \tau^{-1}\alpha})$, whence
\begin{equation}\label{15}
s_{0, \tau\alpha}=\varepsilon_{E/F}(w)s_{0, \alpha}.
\end{equation}
Assume $w\in W_F$ is sent to the chosen generator $\sigma\in
\Sigma$, so $\varepsilon_{E/F}(w)=\zeta$. Recall that $s_0\in
Z(\hat{H})$ is given by block-diagonal matrices $s_{0, \alpha}$ of the
size $b$ with distinct eigenvalues. Equation (\ref{15}) now implies
that the  $s_{0, \alpha}$ can be partitioned into $a'=a/d$ subsets of
the form
\begin{equation}\label{16}
(s_1, \zeta s_1, \cdots, \zeta^{d-1}s_1; s_2,\zeta s_2, \cdots,
\zeta^{d-1}s_2; \cdots );
\end{equation}
 the entries being block-diagonal, we assimilate them to scalars. In
 particular $d\,|\,a$, so  $d\,|\,n$. The scalars $s_j ~(j=1, \cdots,
 a')$ verify $s_j\neq \mu s_{j'}$ for any $\mu\in {\mu}_d(\C)$.  
Equation (\ref{15}) now uniquely determines $\tau_1=\tau(\sigma)$: it
is a product of $a'$ $d$-cycles. 

Consider now an arbitrary element $w\in W_F$. If $w\in W_E$,
$\varepsilon_{E/F}(w)=1$ and equation (\ref{15}) implies that
$\tau=1$, the eigenvalues being distinct. Thus $W_F$ acts via
$\Sigma=W_F/W_E$ and $\tau(w)=\tau_1^{\iota(w)}$. The image of $W_F$
is therefore a cyclic subgroup of $\fS_a$, of order $d$, preserving
the strings of length $d$ in equation (\ref{16}). The
ellipticity of $H$ now implies that this action is transitive, so
$a=d$. 

Since we have determined $n(w)$ up to central elements in $\hat{H}$,
we have now computed the $F$-group $H$. Indeed, $\hat{H}=GL(b)\times
\cdots \times GL(b)$ ($a=d$ factors) and $W_F$ acts via $\Sigma$,
cyclically permuting the factors. This implies that $H$ is isomorphic
to ${\rm Res}_{E/F}(GL(b)/E)$. This embedding into $\hat{G}$ is given
on $\hat{H}$ by 
\[ (h_1, \cdots, h_d)\mapsto {\rm diag}(h_1\oplus \cdots \oplus
h_d)\in GL(n, \C)^d.\]
We can now summarise the main result of this section: 
\begin{proposition}\label{prop:endoscopy}
\begin{enumerate}
\item If $d~{\not\,\mid}~ n$, there exists no elliptic endoscopic group for
  $(\theta, \omega)$. 

\item If $n=db$, there exists (at least) one endoscopic datum for
  $(\theta, \omega)$ given by the foregoing construction. 

\end{enumerate}
\end{proposition}
To complete the proof of the proposition, we still have to show that
we can choose the $z_i(w)$ so as to satisfy equations (E5) and
(\ref{12}). We will obtain in fact a more precise result.

The permutation $\tau$ associated to $w\in W_F$ is a cyclic
permutation on the indices $\alpha$. We have
$\tau(z_{\alpha})\tau^{-1}=(z_{\tau^{-1}\alpha})$. We now assume that
for $\iota(w)=k$, $\tau^{-1}(\alpha)=\alpha-k$. The relation (\ref{4})
now reads
\begin{equation}\label{17}
z_{i,\alpha}(ww')=z_{i,\alpha}(w)z_{i+k,\alpha-k}(w').
\end{equation}
We consider all indices as elements of $\Z/d\Z$. Now fix $\beta$ (mod
$d$). Now equation (\ref{17}) yields for $i+\alpha=\beta$:
\[z_{i,\beta-i}(ww')=z_{i,\beta-i}(w)z_{i+k,\beta-i-k}(w').\]
Set $\zeta_i^{\beta}(w)=z_{i,\beta-i}(w)$: we now have
\[ \zeta_i^{\beta}(ww')=\zeta_i^{\beta}(w)\zeta_{i+k}^{\beta}(w').\]
This means that $\zeta^{\beta}=(\zeta_i^{\beta})_i$ is a $1$-cocycle
of $W_F$ for the action of $W_F$ on $(\C^\times)^d$. However $H^1(W_F,
(\C^\times)^d)$, for  $W_F$ acting by its action on the dual group of
$\mbox{Res}_{E/F}(GL(1))$ is equal to $H^1(W_E, \C^\times)=\mbox{Hom}(C_E,
\C^\times)$ by Shapiro's lemma. Thus we see that each character
$\eta_{\beta}$ of $C_E$ defines such a cocycle, by
$\zeta_i^{\beta}(w)=\eta_{\beta}(\delta_i(w))$. We can then set
\begin{equation}\label{18}
z_{i,\alpha}(w)=\eta_{i+\alpha}(\delta_i(w)),
\end{equation}
and we see that a section (and therefore a subgroup) is defined by the
choice of $\eta_{\beta}$. 

We still have to fulfill the condition given by equation (\ref{14}). 
Assume first that $w\in W_E$. Then $\tau=1$, and the condition is
simply
\[z_{i+1}(w)=a_i(w)z_i(w).\]
Now $z_i=(z_{i,\alpha})$ with
$(z_{i,\alpha})=(\eta_{i+\alpha}(w))_i=(\eta_{i+\alpha}(\sigma^iw))_i$. Here
$\sigma^i$ is the automorphism of $C_E$ obtained by quotienting the
automorphism $w\mapsto \sigma^i w\sigma^{-i}$ of $W_E$. So the
condition is, 
\[
\eta_{i+1+\alpha}(\sigma^{i+1}w)=\omega(\sigma^iw)\eta_{i+\alpha}(^{\sigma^i}\!w).\]
($w$ being seen as an element of $C_E\cdots$), i.e.,
\begin{equation}\label{19}
\eta_{\alpha+1}\circ \sigma=\omega\eta_{\alpha}, \quad
\alpha=1,\cdots, d.
\end{equation}
We can now write the condition of equation (\ref{14}) for
$w=\sigma$. Recall from equation (\ref{7}) that
$(\delta_i(\sigma))=(1,\cdots, 1,\sigma^d)$ with $\sigma^d\in W_E$. We
still have $z_{i, \alpha}(w)=\eta_{i+\alpha}(\delta_i(w))$, whence
\[\begin{split}
z_{i, \alpha}(\sigma)&=1 \quad 1\leq i\leq d-1\\
z_{d, \alpha}(\sigma)&=\eta_{\alpha}(\sigma^d).
\end{split}
\]
Morever, as before
\[ \eta_{\alpha+1}\circ
\sigma(\sigma^d)=\omega(\sigma^d)\eta_{\alpha}(\sigma^d),\]
and $\sigma\sigma^d\sigma^{-1}=\sigma^d, ~\omega(\sigma^d)=\zeta$,
whence 
\begin{equation}\label{20}
z_{d,\alpha} (\sigma)=(\eta, \zeta\eta, \cdots, \zeta^{d-1}\eta)
\end{equation}
where $\eta=\eta_1(\sigma^d)$. Write $z_i = (z_{i,\alpha})= (z_{i,\alpha}(\sigma))$.Then equation (\ref{14}) for $\sigma$
reads:
\begin{equation}\label{21}
\begin{split}
s_0z_2&=z_1\\
z_3&=z_2\\
&.\\
&.\\
z_d&=z_{d-1}\\
^{\tau}\!s_0^{-1}z_1& =\zeta z_d
\end{split}
\end{equation}
and is obviously not satisfied. Recall that for $k=k(\sigma)=1$ we
have set $(^{\tau}\!s_0)_{\alpha}=s_{0, \tau^{-1}\alpha}=s_{0,
  \alpha-1}$ and, cf equation (\ref{15}), 
$s_{0, \tau\alpha}=s_{0, \alpha+1}=\zeta s_{0,\alpha}$. We can choose 
\[s_0=(1,\zeta, \cdots, \zeta^{d-1}). \]
Substituting $s_0~^{\tau}\!s_0^{-1}=\zeta$ in equation (\ref{21}), we
get
\begin{equation}\label{22}
\begin{split}
s_0z_2&=z_1\\
z_3&=z_2\\
&.\\
&.\\
z_d&=z_{d-1}\\
s_0^{-1}z_1& = z_d.
\end{split}
\end{equation}
Now we can replace the $z_i$ by cohomologous elements for the action
of $W_F$ (via $\Sigma$) giving the cocycle relation (\ref{4}), i.e., 
\[ w(z_{i\alpha})=z_{i+k, \alpha-k}.\]
A coboundary is given for $w\in W_F, ~k=\iota(w)$ by
\[ \zeta_{i\alpha}(w)=v_{i\alpha}v_{i+k,\alpha-k}^{-1}.\] 
In particular for $w=\sigma$:
\[ \zeta_{i\alpha}(\sigma)=v_{i\alpha}v_{i+1,\alpha-1}^{-1}.\]
As we did for $w\in W_E$ we can write $\zeta_{i,\alpha}=\zeta_i^{\beta}$ with
$\beta=i+\alpha, ~
\zeta_i^{\beta}=v_{i}^{\beta}(v_{i+1}^{\beta})^{-1}:=u_i^{\beta}$. The
$u_i^{\beta}$ must then satisfy the condition, 
\begin{equation}\label{23}
\prod_{i+\alpha=\beta} u_i^{\beta}=1.
\end{equation}
Our equations then become, with $\alpha=1,\cdots, d$: 
\begin{equation}\label{24}
\begin{split}
s_{0, \alpha} u_2^{2+\alpha}&= u_1^{1+\alpha}\\
u_3^{3+\alpha}&= u_2^{2+\alpha}\\
&.\\
&.\\
u_{d-1}^{d-1+\alpha}&=u_{d-2}^{d-2+\alpha}\\
z_{d,\alpha}u_d^{\alpha}&=u_{d-1}^{d-1+\alpha}\\
s_{0,\alpha}^{-1}u_1^{1+\alpha}&=z_{d,\alpha}u_d^{\alpha}.
\end{split}
\end{equation}
Write $u= u_2 =(u^{\alpha})$. Then d-3 lines
yield: 
\[ u_{i+1}^{1+\alpha}=u_i^{\alpha}, \quad i=2, \cdots, d-2,\]
\begin{equation}\label{25}
\mbox{so} \quad u_i^{\alpha}=u^{\alpha-2+i}\quad  i=2, \cdots, d-1.
\end{equation}
Thus
\[\begin{split} 
s_{0,\alpha}u^{2+\alpha}&=u_1^{1+\alpha}\\
z_{d,\alpha}u_d^{\alpha}&=u^{\alpha+2}\\
s_{0,\alpha}^{-1}u_1^{1+\alpha}&=z_{d,\alpha}u_d^{\alpha}, 
\end{split}
\]
equations obviously compatible. However we must choose the
$u_i^{\beta}$ verifying equation (\ref{23}). We determine $u_1$ and
$u_d$ by the first equations. Thus,
\[u_1^{\alpha}=s_{0,\alpha-1}u^{\alpha+1}, \quad 
u_d^{\alpha}=z_{d,\alpha}^{-1}u^{\alpha+2}.\]
The product is then
\[\begin{split}
\prod_{i=1}^{d-1}u_i^{\beta-i}&=s_{0,\beta-2}z_{d,\beta}^{-1}u^{\beta}
\prod_{i=2}^{d-1}u^{\beta+2-2-i}u^{\beta+2}\\
&=s_{0,\beta-2}z_{d,\beta}^{-1}\prod_j u^  {\beta+j},
\end{split}
\]
where the product is taken over $2\Z/d\Z\subset \Z/d\Z$. However,
\[
s_{0,\beta-2}z_{d,\beta}^{-1}=\zeta^{\beta-3}\zeta^{1-\beta}\eta=\zeta^{-2}\eta\]
is constant, cf. (\ref{20}). It suffices therefore to impose on $u=u_2$
the two conditions (at most)
\[\prod_{j\equiv \beta}u^{\beta+j}=\zeta^2\eta^{-1}\]
for $\beta \in \Z/d\Z$, the congruence being modulo $2\Z/d\Z$. 

We can now define the new cocycle in $Z^1(W_F, Z(\hat{H}))$ by
multiplying the previous map by the coboundary just obtained. It
defines a new section, which verifies the defining condition
(\ref{14}), obviously for $w\in W_E$, for all powers of $\sigma$, and
therefore for $w\in W_F$. This proves the second part of Proposition 
\ref{prop:endoscopy}, and morever it exhibits an explicit section. We
have morever \footnote{ We could avoid this verification by using Lemma 4.5 of Borel\cite{Bo}. It will be clearer explicitly to exhibit the conjugation. }:
\begin{lemma}\label{lemma4.4}
The new section is conjugate in $\hat{G}$ to the section given by the
$z_{i, \alpha}$. In particular they define (up to conjugation in
$\hat{G}$) the same embedding  $^L\!H\to ^L\!G$. 
\end{lemma}
This is clear: if we define by $\bs_1$ the previous section and $\bs_2$
the new one, we have by construction
\[ \bs_1(w)=z^1(w)\tau, \quad  \bs_2(w)=z^2(w)\tau,\]
with $\tau=\tau(w)$, and, with $k=\iota(w)$:
\[z^2(w)=z^1(w)z(\tau z\tau^{-1})^{-1}_{i+k},\]
with $z=(z_i)\in Z(\hat{H})^d$. Thus, 
\[ \bs_2(w)=z^1(w)z \tau (z_{i+k})^{-1}\]
since $\tau$ is diagonal, so
\[ \bs_2(w)=z (z^1(w)\tau)w(z)^{-1}.\]

Recall that $\xi_1$ was defined by 
\begin{equation}\label{25}
\begin{split}
\xi_1 : h &\mapsto \mbox{diag} (h) \quad (h\in \hat{H}\\
w&\mapsto (s_1(w),w)=(z^1(w)\tau, w)\quad (w\in W),
\end{split}
\end{equation}
so $(h, w)\mapsto (hz^1(w)\tau, w)$. Similarly we define $\xi_2$ by 
\[ \begin{split}
\xi_2 : h &\mapsto \mbox{diag} (h) \\
w&\mapsto z(z^1(w)\tau, w)z^{-1}.
\end{split}
\]
Since $z\in Z(\hat{H})^d$, it centralizes $\mbox{diag}(h)$, so $\xi_2$
is conjugate to $\xi_1$. In particular, they will have the same effect
on functions of Hecke matrices, the data for endoscopy. 

Recall that two endoscopic data, 
\[ \G_1'=(G_1', \cG_1',\tilde{s}_1) \quad \mbox{and}\quad 
\G_2'=(G_2', \cG_2',\tilde{s}_2)\]
are equivalent if there exists $g\in \hat{G}$ such that 
\[ g\cG_1'g^{-1}=\cG_2', \quad g\tilde{s}_1'g^{-1}=z\tilde{s}_2\]
for an element $z\in Z(\hat{G})$. (Recall that $\tilde{s}_i=(s_i,
\theta)$ with $s_i\in \hat{G})$. 

\begin{proposition}[Waldspurger]\label{waldspurger}
Assume that $d\,|\,n$. Then there exists only one equivalence class of
endoscopic data for $(G, \theta, a)$. 
\end{proposition}
For the proof we can use our previous construction of a section. It
follows from our analysis that we must take $s_0=(1, \zeta, \cdots,
\zeta^{d-1})$ (in fact block matrices of size $b$) up to  a
scalar. For two choices we therefore have
$g\tilde{s}_1g^{-1}=z\tilde{s}_2$ (take $g\in GL(n)^d$ diagonal). We
can therefore assume that $s_0$ is fixed. Then $\hat{H}\subset
\hat{G}$ is well defined. 

If $\cG'\subset \hat{G}\rtimes W_F$ is an endoscopic subgroup, its
fiber $\cG'_w$ over $w\in W_F$ is equal to
$\hat{G}_{\tilde{s}}\bs(w)$ for any section $\bs$. In particular it is
equal to $\hat{H}\bs(w)$ in our case, with of course $\hat{H}\subset
\hat{G}$ given by the diagonal embedding. Assume $\bs_1$ and $\bs_1'$
are two sections as before. We obtain $\xi_1, \xi_1'$ and
$\xi_2=z\xi_1z^{-1}, ~\xi_2'=z\xi_1'z^{-1}$. The choice of $z$ does
not depend on the choice of the characters $\eta^{\beta}$ in the
construction of $\bs_1, ~\bs_1'$. Then, since $z$ commutes with
$\mbox{diag}(\hat{H})$, we see that the two corresponding subgroups
(given by $\xi_2, ~\xi_2'$) will be the same if $z^1(w)\in
\mbox{diag}(\hat{H})z^2(w)$ for $w\in W$. 

We have seen that $z^1(w)$ could be obtained from the characters
$\eta_{\alpha}$, with $z_{i, \alpha}(w)=\eta_{i+\alpha}(\delta_i)$ and
$\eta_{\alpha+1}\circ \sigma=\omega \eta_{\alpha}$. If
$(\eta'_{\alpha})$ are another choice,
$\eta'_{\alpha}=\eta_{\alpha}\nu_{\alpha}$ with $\nu_{\alpha+1}\circ
\sigma=\nu_{\alpha}$. This implies that
$\nu_{\alpha+i+1}(\delta_{i+1})= \nu_{\alpha+i}(\delta_i)$, and the
correction belongs to $\mbox{diag}(Z(\hat{H}))$. This concludes the
proof \footnote{Waldspurger's proof, terser, was more elegant.}. 

\section{ proof of theorem 1}
\subsection{}
We now have to understand the effect of our homomorphism $\xi_1$ as in
Equation (\ref{xi}) of $L$-groups on the data pertinent to the
stabilization, i.e., on the data composed of Hecke matrices for almost
all primes. Note that 
\[^L\!H=GL(b,\C)^d\rtimes W_F\]
is not a direct product, so a ''Hecke matrix'' at a prime $v$ is in fact a
conjugacy class in $\hat{H}\times \mbox{Frob}_v$ under the conjugation
action of  $\hat{H}$. 

For this we must first consider a simple case. Assume that $\omega=1$,
so we are in the case of non-twisted base change, i.e., characterizing
the representations $\Pi_n$ of $GL(n, \A_E)$ such that
$^{\sigma}\!\Pi_n\simeq \Pi_n$. 

Recall (\cite{AC}, \cite{C}) the two natural operations associated
to (cyclic) base change. The first is automorphic restriction, denoted earlier by $BC_F^E$, sending
representations of $GL(n, \A_F)$ to representations of $GL(n,
\A_E)$. It is associated to the diagonal embedding
\[ ^L\!G_0\to ^L\!G\]
where $^L\!G_0$ is the $L$-group of $GL(n)/F$, $^L\!G$ the $L$-group
of $\mbox{Res}_{E/F}(GL(n)/E)$, so $\hat{G}=GL(n, \C)^d$: 
\[(g,w)\mapsto (\mbox{diag}(g), w)\quad (w\in W_F).\]
Suppose $n=db$. The second operation is automorphic induction, denoted by $AI_E^F$, sending
representations of $GL(b,\A_E)$ to those of $GL(n,\A_F)$. The
associated embedding of $L$-groups is given by
\[ (g_1, \cdots, g_d)\mapsto g_1\oplus\cdots\oplus g_d\quad (g_i\in
GL(b,\C))\]
\[\mbox{and}\quad (1,w)\mapsto (\tau(w), w)\]
where $\tau(w)\in \fS_d$ (realized as before by block-scalar matrices
in $GL(n, \C)$), and 
\[ \tau(\oplus g_i)\tau^{-1}=(g_{i+k}),\]
where $k=i(w)$, so $\tau^{-1}i=i+k$.

We simply write $Res$ and $Ind$ for these two operations, the fields being here $F$ and $E$. This corresponds to our constructions in (\cite{AC}), taking
$\omega=1$. 
The corresponding operations are described in \cite[Chapter
3]{AC}, see also \cite {C}. They are well-defined for representations that are 'induced
from cuspidal' (\cite[Sections 3.1, 3.6]{AC}), i.e., induced from unitary cuspidal representations.  

Composing these two operations, we get a homomorphism of $L$-groups 
$\mu_1: GL(b)^d\rtimes W_F\to GL(n)^d\rtimes W_F$, given by, 
\begin{equation}\label{II.1}
\begin{split}
(g_1, \cdots, g_d)& \mapsto \mbox{diag}(g_1\oplus\cdots\oplus g_d) \\
(1,w) &\mapsto (\mbox{diag}~\tau(w), w).
\end{split}
\end{equation}

Recall also that for $\pi_i ~(i=1,\cdots, r)$ representations of
$GL(n_i, \A_E)$, there is an associated representation $\boxplus
\pi_i$ of $GL(n,\A_E)$ obtained by parabolic induction ($n=\sum
n_i$). We recall the following well known result:
\begin{proposition}\label{prop:resind}
For $\pi_b$ a representation of $GL(b,\A_E)$ induced from cuspidal, 
\[ \mbox{Res}\circ \mbox{Ind} (\pi_b)=\pi_b\boxplus
\sigma\pi_b\boxplus\cdots\boxplus \sigma^{d-1}\pi_b.\]
\end{proposition}
Consider finite primes $w~|~v$ of $E$ over $F$ where all data are
unramified. If $t_{w'}$ is the Hecke matrix of $\pi_b$ at such a prime
$w'$, the matrix $T_v$ of $\mbox{Ind}(\pi_b)$ at $v$ is 
\[ T_v=\bigoplus_{w'~|~v}(t_{w'}^{1/f}\oplus  \zeta t_{w'}^{1/f}
\oplus \cdots\oplus \zeta^{f-1}t_{w'}^{1/f}),\]
where $\zeta$ is a primitive root of unity of order $f=[E_w:F_v]$. 
The Hecke matrix $T_w$ of $\mbox{Res}(\Pi)$ for a representation $\Pi$
of $GL(n,\A_E)$ is $T_v^f$. Thus the Hecke matrix of $\mbox{Res}\circ
\mbox{Ind} (\pi_b)$ at a prime $w$ is 
$\bigoplus_{w'~|~v}(t_{w'}\oplus   \cdots\oplus t_{w'})$, equal to
$\bigoplus_{\sigma\in Gal(E/F)}t_{\sigma w}$, the Hecke matrix of
the right-hand side. Since the representations on the two sides of the
equality are induced from cuspidal, they are equal. 

Now there exists an obvious homomorphism of $L$-groups realising the
operation
\[ \pi_b\mapsto \pi_b\boxplus
\sigma\pi_b\boxplus\cdots\boxplus \sigma^{d-1}\pi_b.\]
First $(\pi_1, \cdots, \pi_d)\mapsto \pi_1\boxplus\cdots\boxplus \pi_d$
is given by $ GL(b)^{d^2}\to GL(n)^d$,
\begin{equation}\label{II.2}
(g_{ki})\mapsto \left(\bigoplus_kg_{ki}\right)_i.
\end{equation}
It is obviously compatible with the operation of the Weil group. On
the other hand $\pi\mapsto ^{\sigma}\!\pi$ is given by 
\[(g_i)\mapsto (g_{i+1}).\]
So the composite operation is given by 
\[ \begin{split}
GL(b)^d &\to GL(n)^d\\
(g_i)&\mapsto  \left(\bigoplus_k g_{i+k}\right)_i.
\end{split}\]
It is equivariant for the action of $W_F$ acting (via the restriction
of scalars) on both sides. Thus we get
\begin{equation}\label{II.3}
\begin{split}
\mu_0 :~ GL(b)^d\rtimes W_F& \to GL(n)^d\rtimes W_F\\ 
 (g_i, w)&\mapsto \left(\left(\bigoplus_k g_{i+k}\right)_i, w\right).
\end{split}
\end{equation}
Since the two homomorphisms of $L$-groups $\mu_0$ and $\mu_1$ have the
same effect on representations, they should be conjugate by an element in
$\hat{G}=GL(n)^d$. We proceed to exhibit this conjugation. We first
consider the connected dual groups. We seek $P=(P_i)\in GL(n)^d$ such
that 
\[ P\mu_0(g)P^{-1}=\mu_1(g), \quad g=(g_i)\in GL(b)^d.\]
Thus, 
\[ P_i(\bigoplus_k g_{i+k})P_i^{-1}=g_1\oplus\cdots\oplus g_d,\]
for each $i$. If $Q$ is (a block-scalar matrix) associated to a
permutation $\tau\in \fS_d$, 
\[ Q(\bigoplus_k g_{k})Q^{-1}=(g_{\tau^{-1}(k)}).\]
Thus $P_i$ must be the permutation matrix associated to $\tau$, where
$\tau(k)=i+k$. 

To avoid confusion, we now replace our indices $k$ by $\alpha$ (in
conformity with the previous section) and write $k=\iota(w), ~(w\in
W_F)$. The conjugation of the homomorphisms on $W_F$ gives, 
\[(P,1)(1,w)(P,1)^{-1}=(\mbox{diag}\tau(w), w)\quad (w\in W_F),\]
\begin{equation}\label{II.4}
\mbox{so}\quad P_iP_{i+k}^{-1}=\tau(w). 
\end{equation}
The left hand side is associated to $\tau(\alpha)=\alpha-k$. Thus
(with the previous choices) we must take
\[ \tau(w)=\tau_1^{\iota(w)}, \quad \tau_1(\alpha)=\alpha-1.\]
(compare with the formula for $\tau$ preceding (\ref{II.1})). 

We have therefore proved:
\begin{lemma}\label{lemma:II.1}
With the above notation, 
\[  P\mu_0P^{-1}=\mu_1,\]
where $P=(P_i)\in GL(n,\C)^d$ and $P_i$ is the permutation matrix
associated to $\alpha\mapsto i+\alpha$. 
\end{lemma}

\subsection{}
Now return to the homomorphism $\mu_0$ (\ref{II.1}) realising the
operation
\[ \pi\mapsto \boxplus_{\alpha=1}^d\sigma^{\alpha}\pi.\]
($\pi$ being a representation of $GL(d,\A_E)$). Let $\eta_1, \cdots,
\eta_d$ be characters of $C_E$, associated to the parameters
$\eta_{\alpha}(\delta_i(w))\in (\C^\times)^d$ as discussed before equation
(\ref{18}). Now the homomorphisms (\ref{II.2}) can be multiplied by
the homomorphisms associated with the $\eta_{\alpha}$, so we see that 
\[ (g_{\alpha i})\mapsto \left(\bigoplus_{\alpha}g_{\alpha
    i}\eta_{\alpha}(\delta_i(w))\right)_i\]
corresponds to $(\pi_{\alpha})\mapsto \boxplus_{\alpha} \pi_{\alpha}\otimes
\eta_{\alpha}.$
In particular, $\pi\mapsto \boxplus_{\alpha}\sigma^{\alpha} \pi_{\alpha}\otimes
\eta_{\alpha}$, is then given by 
\[
\begin{split}
\xi_{0}: GL(b)^d\rtimes W_F& \to GL(n)^d\rtimes W_F \\
(g_i,w)&\mapsto [(g_{\alpha+i}\eta_{\alpha}(\delta_i(w))_i, w].
\end{split}
\]
Conjugating by $P$, we obtain a homomorphism $\xi_1$. On $GL(b)^d$, it
coincides with $\eta_1$. For $w\in W_F$ we must compute
\[(P,1)[\left(\bigoplus_{\alpha}\eta_{\alpha}(\delta_i(w))\right)_i, w](P^{-1}, 1).\]
The $i$-th component is
\[P_i\left(\bigoplus_{\alpha}\eta_{\alpha}(\delta_i(w))\right)P^{-1}_{i+k}
= P_i\left(\bigoplus_{\alpha}\eta_{\alpha}(\delta_i(w))\right)P^{-1}_{i}\tau(w)\]
by (\ref{II.4}). The conjugation by $P_i$ is the permutation $\alpha
\mapsto i+\alpha$, so this is 
\[ \left(\bigoplus_{\alpha}\eta_{i+\alpha}(\delta_i(w))\right)\tau(w).\]
In conclusion,
\begin{lemma}\label{lemma:II.2}
The map $\pi\mapsto \boxplus_{\alpha}\sigma^{\alpha} \pi_{\alpha}\otimes
\eta_{\alpha}$ is realized by the homomorphism, 
\[
\begin{split}
\xi_{1}: GL(b)^d\rtimes W_F& \to GL(b)^d\rtimes W_F\\
\xi_{1}(g_1,\cdots, g_d)&=\mbox{diag}(g_1\oplus\cdots \oplus g_d)\\
\xi_1(w)&=[\left(\left(\bigoplus_{\alpha}\eta_{i+\alpha}(\delta_i(w))\right)\tau(w)\right)_i,
w].
\end{split}
\]
\end{lemma}
We note that this is the homomorphism $\xi_1$ obtained from endoscopy
as in Equations (\ref{25}) and (\ref{18}). 

\subsection{}
We can now complete the proof of Theorem 1. Assume first $d$ does not
divide $n$. We then have 
\[T_{disc}(\phi \times \theta; \omega^{-1})=0.\]
The cuspidal representation occuring in the discrete trace have
multiplicity one, and their families of Hecke eigenvalues (away from a
finite set $S$ of primes) are linearly independent, and independent
from those of other representations. Of course only the cuspidal
representations such that $^{\sigma}\!\Pi\simeq \Pi\otimes \omega$
contribute. We conclude that there are no such representations, as was of course clear from the consideration of the central characters (see the remark after Theorem 1.)

Consider now the case when $d$ divides $n$. There is only, up to
equivalence, one endoscopic datum $H$; if $\pi_b$ is a cuspidal
representation of $GL(b, \A_E)$, the associated map on Hecke matrices
sends (up to conjugation) $t_w(\pi_b)$ to 
\[ t_w(\pi_b\otimes \eta_1)\oplus 
t_w(\sigma\pi_b\otimes \eta_2)\oplus\cdots \oplus t_w(\sigma^{d-1}
\pi_b\otimes \eta_d)\]
at primes where all data are unramified, as follows from the
conjugation of $\xi_1$ and $\mu_0$. This then remains true if $\pi_b$
is any automorphic representation of $GL(b,\A_E)$, in particular for
those appearing in the discrete trace formula for $H$. Here again, the sum
\[\sum_{^{\sigma}\!\Pi\simeq \Pi\otimes
  \omega}\mbox{trace}(I_{\theta}(\Pi\otimes \omega^{-1})(\phi))\]
over the cuspidal representations yields, evaluated against a function
$\phi$, a linear combination of characters of the Hecke algebra
linearly independent from all others, associated to induced
representations. The same argument then proves the Theorem.

\end{document}